\documentclass{amsart}
\usepackage{latexsym} % for \Box
\usepackage{amsmath,amssymb,amsfonts,amsthm}
\usepackage{verbatim}
\usepackage[dvips]{graphicx}
\usepackage{textcomp}
\usepackage{graphicx,color}
\usepackage[normalem]{ulem}

\newtheorem{remark}{Remark}
\newtheorem{theorem}{Theorem}
\newtheorem{proposition}{Proposition}
\newtheorem{corollary}{Corollary}

\newtheorem{lemma}{Lemma}
\newtheorem{example}{Example}

\newcommand{\me}{\mathbb{E}}
\newcommand{\mn}{\mathbb{N}}
\newcommand{\mr }{ \mathbb{R}}

\newcommand{\R}{\mathbb{R}}
\newcommand{\Prob}{\mathbf P}

\newcommand{\F}{\mathcal F}

\newcommand{\g}{\mathfrak g}
\newcommand{\h}{\mathfrak h}

\newcommand{\C}{\mathcal C}
\newcommand{\mw}{\mathbf{w}}

\title{Invariance of 0-currents under diffusions.}
\author{ Diego S. Ledesma }
\address{Departamento de
 Matem\'{a}tica, Universidade Estadual de Campinas \\ 13.081-970 -
 Campinas - SP, Brazil. \\ E-mail: dledesma@ime.unicamp.br}
 \author{Fabiano B. da Silva}
\address{Departamento de Matem\'{a}tica, Universidade Estadual
Paulista J\'{u}lio de Mesquita Filho \\ Bauru - SP, Brazil. \\ E-mail: fabiano@fc.unesp.br}

\begin{document}
\maketitle
\begin{abstract}
We study two actions of a stochastic flow $\varphi_t$  on the space of $0-$currents $T$ of a differentiable manifold $M$. In particular, we give conditions on a current $T$ to be invariant under these actions. Also, we apply our results to the case of $0-$currents that come from harmonic measures associated to certain type of foliations on homogeneous spaces.
\end{abstract}

\section{Introduction.}
In dynamical systems some evolution problems can be related to the evolution of geometrical currents over a differentiable manifold. For example, consider the following situation: let $D$ be a domain in $\mr^n$ and suppose that we want to study the evolution of the volume of $D$ under a flow $\varphi_t:\mr^n\rightarrow \mr^n$. This problem can be studied from different points of views. In \cite{mcdonald}, Kinateder and McDonald used the fact that $D_t=\varphi_t(D)$ is a stochastic process on a Frechet manifold, and developed a stochastic calculus for that setup to study the behavior of the volume $vol(D_t)$. 

Another approach to the same problem, can be given by a simple use of the change of variables formula. In fact, we observe that 
\[
 vol(D_t)=\int_{D_t}~dx=\int_D\varphi^*_tdx.
\]
Therefore, the study of the evolution of $vol(D_t)$ turns into the study of the evolution of the geometric current $T_t$ induced by the stochastic $n-$form $\varphi_t^*dx$  on $\mathbb{R}^n.$ The main advantage of making this is that we change a structure of a Frechet manifold by a vector space. Then, our object of study is a stochastic process  $T_t$ on a vector space.

In this work we  study the natural class of stochastic processes on the space of currents which are induced by stochastic processes in the base space. Our main set up is the following: let $(\Omega,\mathcal{F},\mathbb{P})$ be a probability space, let $\varphi_t$ be a stochastic flow on a differentiable manifold $M$, and let $C^\infty(M)$ be the set of smooth real-valued functions $f : M \rightarrow \mr$ with compact supports.  We denote by $\C^0(M)$ to the space of $0-$currents, i.e. the space of linear functionals $T:C^\infty(M)\rightarrow \mr$. The flow $\varphi_t$ acts naturally on $\C^0(M)$ by
\[
 \varphi_t^*T(f)=T(f\circ \varphi_t),
\]
for every $f\in C^\infty(M)$. We observe that $\C^0(M)$ is an infinite dimensional vector space and that $\varphi_t^*T$ is a random variable on $\C^0(M)$.
Thus, it is possible to take means and we are able to define another action on the space of currents by
\[
 (\varphi_t\star T)(f)=\me[\varphi_t^*T(f)].
\]
Inspired in the definition above we say that $\varphi_t\star T$ is the mean of the random variable
$\varphi_t^*T$.

A current $T$ will be called invariant under $\varphi_t$ if \[\varphi_t^*T=T,\] and  invariant in mean by $\varphi_t$ if
\[
  (\varphi_t\star T)(f)=T(f),\hspace{1cm}\forall f\in C^\infty(M).
\]
When the current $T$ is defined in terms of a borel measure $\mu$ on $M$,
 i.e.
 \[
  T(f)=\int_Mf~d\mu,
 \]
for every $f\in C^\infty(M)$, the definition of invariance under
$\varphi_t$ coincides with the usual definition of the measure $\mu$ being invariant
under $\varphi_t$. Also, for the case of borel measures, the notion of invariance in mean is equivalent
to the invariance of $\mu$ under the heat semigroup generated by
$\varphi_t$ (see e.g. Klieman \cite{klieman})

It is easy to see that a invariant current is invariant in mean. In fact, for an invariant current $T$ we have
\begin{eqnarray*}
 \me[\varphi_t^*T(f)]&=&\me[T(f)]\\
 &=&T(f).
\end{eqnarray*}
But the converse is hard to happen. We will give
conditions that characterize the currents that are invariant and invariant in
mean under diffusions. In particular, we will study the case of $0-$currents given
by volume forms as a natural consequence of its theory and, as an application, we will specialize on the particular case of harmonic measures of foliations over homogeneous spaces. 

The article is organized as follows: in Section 2, we study the action of a diffusion over a current $T$ and obtain conditions to guarantee the different types of invariance. In Section 3, we work in the specific case of currents defined by $n-$forms. Particularly, we study the problem discussed above but now on the set up of domains in Riemannian manifolds . Finally, in Section 4, we apply the above results to the case of foliations on a homogeneous space.

\section{Invariance under stochastic flow generated by SDE. }

Here, we start our study of the actions over a current $T$ of a diffusion defined by a SDE on $M$. In particular, we obtain conditions to guarantee the different types of invariance. Our main tools are the convergence theorems of de Rham (see \cite{derham} pg 68) and the It\^o formula for the flow of a SDE given by Kunita in \cite[pg 263]{Kunita-1}.

Let $(M,g)$ be a compact Riemannian manifold. Consider a stochastic flow $\varphi_t$ generated by a Stratonovich SDE on a manifold $M$:
 \begin{eqnarray}
 dx_t&=&\sum_{i=0}^mX_i(x_t)~\circ dB^i_t\label{EDE1}\\
x_0&=&x,\nonumber
\end{eqnarray}
where $B^0_t = t$, $ (B^1_t, \ldots , B^m_t)$ is a Brownian motion
in $\R^m$ constructed on a filtered probability space $(\Omega,
\F, \F_t, \Prob)$ and $X_0, X_1, \ldots , X_m$ are smooth vector
fields over $M$. We assume that there exist a stochastic solution
flow of diffeomorphisms $\varphi_t:\mr\times
M\times \Omega \rightarrow M$. In particular, when $M$ is compact the flow solution always exists, see e.g. Elworthy \cite{Elworthy} or Ikeda and
Watanabe \cite{IW} among others.
\begin{proposition}
 Let $T$ be a $0-$current. Then
 \begin{eqnarray*}
  T\left(\int_0^tf(\varphi_t)~_\circ dB_t\right)&=&\int_0^tT\left(f(\varphi_t)\right)~_\circ dB_t.
  %
  %\\ T\left(\int_0^tf(\varphi_t)~ dB_t\right)&=&\int_0^tT\left(f(\varphi_t)\right)~ dB_t\\
 % T\left(\int_0^tf(\varphi_t)~dt\right)&=&\int_0^tT\left(f(\varphi_t)\right)~dt
 \end{eqnarray*}

\end{proposition}
\begin{proof} We first observe that $\int_0^tf(\varphi_t)~_\circ dB_t$, 
	%~$\int_0^tf(\varphi_t)~dB_t$ and $\int_0^tf(\varphi_t)~dt$
is a random function on $M$. By a theorem of de Rham (see \cite{derham} pg 68) we have that there are smooth n-forms $\{\nu_n\}_{n\in \mn}$ such that the currents $T_{n}$, defined by
\[
 T_n(f)=\int_Mf~\nu_n,
\]
satysfy $T_n(f)\rightarrow T(f)$ for every $f\in C^\infty(M).$
Since
 \begin{eqnarray*}
  T_n\left(\int_0^tf(\varphi_t)~_\circ dB_t\right)&=&\int_0^tT_n\left(f(\varphi_t)\right)~_\circ dB_t
  %
  %\\
  %T_n\left(\int_0^tf(\varphi_t)~ dB_t\right)&=&\int_0^tT_n\left(f(\varphi_t)\right)~ dB_t\\
  %T_n\left(\int_0^tf(\varphi_t)~dt\right)&=&\int_0^tT_n\left(f(\varphi_t)\right)~dt
 \end{eqnarray*}
 for each $n$, by Fubini's theorem, we have that
\begin{eqnarray*}
 T\left(\int_0^tf(\varphi_t)~_\circ dB_t\right)&=&\lim_{n\rightarrow \infty}T_n\left(\int_0^tf(\varphi_t)~_\circ dB_t\right)\\
 &=&\lim_{n\rightarrow \infty}\int_0^tT_n\left(f(\varphi_t)\right)~_\circ dB_t\\
 &=&\int_0^tT\left(f(\varphi_t)\right)~_\circ dB_t.
\end{eqnarray*}

\end{proof}

In a similar way, we obtain  
\[T\left(\int_0^tf(\varphi_t)~ dB_t\right)=\int_0^tT\left(f(\varphi_t)\right)~ dB_t,\] and  
\[T\left(\int_0^tf(\varphi_t)~dt\right)=\int_0^tT\left(f(\varphi_t)\right)~dt.\]

Given a vector field $X$ over $M$ and a $0-$current $T$ the derivative $XT$ is a $0-$current such that
\[
 XT(f)=-T(Xf),
\]
cf. de Rham \cite{derham} pg 46.
\begin{theorem}
 A $0-$current $T$ is invariant under $\varphi_t$ if and only if \[X_iT=0,\] for all $i=0,\ldots ,m$.
 A $0-$current $T$ is invariant in mean under $\varphi_t$ if and only if
\[\left(X_0 -\frac{1}{2}\sum_{i=1}^m X_i^2\right)T=0.\]
\end{theorem}
\begin{proof}
 By It\^o formula for $\varphi_t^{-1}$ (see \cite[pg 263]{Kunita-1}), we obtain that
 \begin{eqnarray*}
  \varphi_t^{-1*}T(f)&=&T(f\circ \varphi_t^{-1})\\
  &=&T \bigg( f- \sum_{i=1}^m \int_0^tX_i(f\circ \varphi_s^{-1})~ dB^i_s - \int_0^tX_0(f\circ\varphi_s^{-1})~
  ds \\ 
& &  +\int_0^t \frac{1}{2} \sum_{i=1}^m X_i^2(f\circ\varphi_s^{-1})~ds \bigg)\\
  &=&T(f)-\sum_{i=1}^m \int_0^tT(X_i(f\circ \varphi_s^{-1}))~dB^i_s - \int_0^tT(X_0(f\circ\varphi_s^{-1}))ds \\
  & & +\int_0^t\frac{1}{2} \sum_{i=1}^m T(X_i^2(f\circ\varphi_s^{-1}))~ds.
 \end{eqnarray*}
Thus, if  $\varphi_t^*T(f)=T(f)$ then \[\varphi_t^{-1*}T(f)=T(f),\] and therefore $X_iT=0$ for $i=0, \ldots, m$. On the other hand, if $T(X_if)=0$ then $\varphi_t^{-1*}T(f)=T(f)$, and therefore $\varphi_t^*T(f)=T(f)$ for every $f$.

To see the other case, we observe that 
\begin{eqnarray*}
  \me[\varphi_t^{-1*}T(f)] &=&T(f)+0+\int_0^t\me\left[T\left(- X_0  (f\circ\varphi_s^{-1}) + \frac{1}{2} \sum_{i=1}^m X_i^2(f\circ \varphi_s^{-1})\right)\right]~ds.
 \end{eqnarray*}
Thus, if $\left( X_0 -\frac{1}{2}\sum_{i=1}^m X_i^2\right)T=0$ we obtain that $\me[\varphi_t^{-1*}T(f)]=T(f)$. Then,
\begin{eqnarray*}
 \me[\varphi_t^*T(f)]&=&\me[ T(f\circ\varphi_t)]\\
 &=&\me[\me[\varphi_t^{-1*}T(f\circ \varphi_t)]]\\
 &=&T(f).
\end{eqnarray*}
Analogously, if $\me[\varphi_t^*T(f)]=T(f)$ then $\me[\varphi_t^{-1*}T(f)]=T(f)$, and therefore \[\left(X_0- \frac{1}{2} \sum_{i=1}^m X_i^2\right)T=0.\]

 %By \cite[Thm 1.2, pg 260]{Kunita-1},
 %\begin{eqnarray*}
 % \varphi_t^*T(f)&=&T(f\circ \varphi_t)\\
 % &=&T\left(f+ \sum_{i=0}^m \int_0^tX_i(f\circ \varphi_s)~ \hat{d}B^i_s+\int_0^tX_0(f\circ\varphi_s) ~ ds +\int_0^t \frac{1}{2} \sum_{i=1}^m %X_i^2(f\circ\varphi_s)~ds\right)\\
 % &=&T(f)+\int_0^tT(X_i(f\circ \varphi_s))~\hat{d}B^i_t+\int_0^tT(X_0(f\circ\varphi_s))+\int_0^t\frac{1}{2} \sum_{i=1}^m T(X_i^2(f\circ\varphi_s))~ds
 %\end{eqnarray*}
%Thus $\varphi_t^*T(f)=T(f)$ iff for each $i$ we have that
%$T(X_if)=0$ for every $f$, i.e. $X_iT=0$ for all $i=0\ldots m$.
%
%
%To see the other case, we just apply the It\^o
%formula to obtain
%\begin{eqnarray*}
%  \me[\varphi_t^*T(f)]&=&\me[T(f\circ \varphi_t)]\\
%  &=&\me\left[T\left(f+\sum_{i=1}^m \int_0^tX_i(f\circ \varphi_s)~\hat dB^i_t+\int_0^tX_0(f\circ \varphi_s) + \frac{1}{2}  \sum_{i=1}^m X_i^2(f\circ %\varphi_s)~dt\right)\right]\\
%  &=&T(f)+0+\int_0^t\me\left[T\left(X_0(f\circ\varphi_s) + \frac{1}{2} \sum_{i=1}^m X_i^2(f\circ \varphi_s)\right)\right]~dt\\
% \end{eqnarray*}

\end{proof}

\begin{remark}
                 Although we are dealing here with $0-$currents it is easy to see that the same kind of calculations can be done for $k-$currents with $0<k\leq n$, since the main tools used here, the de Rham theorems and Kunita's formulas, remain valid for differential (n-k)-forms, with Lie derivatives, instead of the natural vector field derivative on functions.                  
                 In particular, Ustunel in \cite{ustunel} obtained the formula above, for the case of $n-$currents in $\mr^n$, on the space of distributions on $\mr^n$. 
                \end{remark}

\section{Currents defined by a n-forms.}

General results on $0-$currents can be difficult to be obtained if we have no information about its structure. When a $0-$current is defined by an $n-$form we can apply the differential calculus to study it. In this section we will see that the invariance of a $0-$current defined by an $n-$form can be characterized in terms of usual differential operators.

Consider, over an orientable compact Riemannian $n-$manifold $(M,g)$,
the volume form $\mu_g$ defined by the metric $g$. It is well
known that any other $n-$form $\nu$ can be written as $\nu=f\mu_g$
for a smooth function $f$.

The volume form allows us to introduce the concept of divergence of a vector field and Jacobian of a transformation in the following way. The divergence of the vector field $X$ with respect to $\mu$ is the function $div_{\mu}(X)$ such that
\[
 L_X\mu=(div_\mu(X))~\mu.
\]
The Jacobian of a diffeomorphism $\phi$ is the smooth function
$J_\mu\phi$ such that
\[
 \phi^*\mu=J_\mu\phi~\mu.
\]
In particular, when we consider the diffeomorphism
$\varphi_t$ arising from the stochastic
case, we observe that
\begin{eqnarray*}
 \varphi_t^{-1*}T_{\mu_g}(f)&=&\int_Mf\circ\varphi_t^{-1}~\mu_g\\
 &=&\int_{M}f~\varphi_t^{*}\mu_g\\
 &=&\int_Mf~J\varphi_t~\mu_g\\
 &=&T_{\mu_g}(J\varphi_t~f).
\end{eqnarray*}
So, the evolution of $ \varphi_t^{-1*}T_{\mu_g}$ can be studied through the Jacobian of $\varphi_t$. In
particular, if $J\varphi_t=1$ then the current $T_{\mu_g}$ is
invariant under $\varphi_t.$

It is well known (see \cite{Kunita-1},
Theorem 2.1 and Theorem 4.2) that the following formula holds:
\begin{eqnarray*}
\varphi_t^*\omega-\omega&=& \sum_{i=0}^m
\int_0^t\varphi_s^*L_{X_i}\omega~_\circ dB_s^i.
%\\
% &=&\textcolor{blue}{\sum_{i=1}^m} \int_0^tL_{X_i}\varphi_s^*\omega~dB_s^i+\int_0^t\left(L_{X_0}+\frac{1}%{2}\sum_{i=1}^mL_{X_i}^2\right)\varphi_s^*\omega~ds
\end{eqnarray*}

%\textcolor{blue}{coloquei $J_{\mu_g}\varphi_t$ onde estava apenas
%$J\varphi_t$ e o mesmo para $\mu_g$ no lema abaixo!!}

\begin{lemma}
The Jacobian $J_\mu\varphi_t(x)$ satisfy the formula
\begin{eqnarray}
 J_\mu\varphi_t(x)&=& 1+\sum_{i=0}^m \int_0^t   ~divX_i(\varphi_s(x)) ~ (J_\mu\varphi_s(x)) \circ dB^i_s \label{EDEdiv}.
\end{eqnarray}
\end{lemma}
\begin{proof}
 \begin{eqnarray*}
  J_{\mu_g}\varphi_t~\mu_g-\mu_g&=&\varphi_t^*\mu_g-\mu_g\\
  &=&\sum_{i=0}^m\int_0^t\varphi_s^*L_{X_i}\mu_g~_\circ dB_s^i\\
  &=&\sum_{i=0}^m\int_0^t\varphi_s^*(div_{\mu_g}(X_i)\mu_g)~_\circ dB_s^i\\
  &=&\sum_{i=0}^m\int_0^tdiv_{\mu_g}(X_i)\circ \varphi_s~\varphi_s^*\mu_g~_\circ dB_s^i\\
  &=&\sum_{i=0}^m\int_0^tdiv_{\mu_g}(X_i)\circ \varphi_s ~J_{\mu_g}\varphi_s~\mu_g~_\circ dB_s^i\\
  &=&\left(\sum_{i=0}^m\int_0^tdiv_{\mu_g}(X_i)\circ \varphi_s ~J_{\mu_g}\varphi_s~_\circ dB_s^i\right)~\mu_g.
 \end{eqnarray*}
Therefore,
 \begin{eqnarray*}
  J_{\mu_g}\varphi_t-1&=&\left(\sum_{i=0}^m\int_0^tdiv_{\mu_g}(X_i)\circ \varphi_s ~J_{\mu_g}\varphi_s~_\circ dB_s^i\right).
 \end{eqnarray*}

\end{proof}

Thus, we see that if $div_{\mu_g}(X_i)=0$ then $T_{\mu_g}$ is invariant under $\varphi_t$. This result can be generalized as in the following theorem.

\begin{theorem} Let $T_\nu$ be a $0-$current defined by a $n-$form $\nu=f\mu_g$.  Then $T_\nu$ is invariant under $\varphi_t$  if and only if
\[div_{\mu_g}(f X_i)=0, \]for $i=0,..,m$.
\end{theorem}
\begin{proof}
As we saw before, $T_\nu$ is invariant under $\varphi_t$  if and
only if $X_iT_\nu=0$ for $i=0,..., m$. For any
smooth function $h$, we have that
\begin{eqnarray*}
 X_iT_\nu(h)&=&-T_\nu(X_ih)\\
 &=&-\int_MX_ih~\nu\\
 &=&\int_M h~L_{X_i}\nu,
\end{eqnarray*}
%\textcolor{green}{$div(hX)=Xh+h divX$..and....$\int_M div(hX_i)
%~\nu=0$ because $M$ is compact!!!}
where we used Stokes's theorem and Cartan's formula for $L_X.$
Since,
\begin{eqnarray*}
 div_{\nu}(X_i)&=&div_{\mu_g}(X_i)+\frac{X_if}{f},\\
 div_{\mu_g}(fX_i)&=&f~div_{\mu_g}(X_i)+X_if,
\end{eqnarray*}
we obtain that
\begin{eqnarray*}
 X_iT_\nu(h)&=&\int_Mh~L_{X_i}\nu\\
 &=&\int_Mh~div_\nu(X_i)~\nu\\
 &=&\int_Mh\left(fdiv_{\mu_g}(X_i)+X_if\right)~\nu \\
 &=&\int_Mh~div_{\mu_g}(fX_i)~\nu.
\end{eqnarray*}
Thus, $X_iT_\nu = 0$ iff  $div_{\mu_g}(fX_i)=0$.
\end{proof}

\begin{corollary}
The current $T_{\mu_g}$ is invariant under $\varphi_t$  if and
only if $div_{\mu_g}(X_i)=0$, for $i=0,...,m$.
\end{corollary}

\begin{example}
Let $M^{2n}$ be a smooth manifold endowed with a symplectic
form $\mw$, and $h_i:M
\rightarrow \R$, $i=0,1,...,k$, be a smooth functions. Let
$X_{h_i}$ the Hamiltonian vector field associated to $h_i$, for
each $i=0,1,...,k$, i.e. the unique vector field that $i_{X_{h_i}}
\mw = dh_i$. Supposing that $M$ is compact, or at least that
each $X_{h_i}$ is complete, let $\phi_t:M \rightarrow M$ be the
one-parameter family of diffeomorphisms generated by the equation
$$dx_t= X_{h_0} (x_t)~ dt + \sum_{i=1}^k X_{h_i}(x_t) _\circ dB_t.$$
The Lioville measure $\mu$ (or symplectic measure) of $M$ is given
by
$$\mu= \frac{\mw^n}{n!} ~.$$
Notice that
$$L_{X_{h_i}} \mu = \frac{1}{n!}\sum_{j=1}^n \mw \wedge \cdots \wedge \underbrace{L_{X_{h_i}} \mw}_{j} \wedge \cdots \wedge \mw ~.$$
Since $$L_{X_{h_i}}\mw= d(i_{X_{h_i}}\mw)= ddh=0,$$ we have
that $L_{X_{h_i}} \mu=0$, and thus $div_{\mu}(X_{h_i})=0$, for
$i=0,1,...,k$. In this way, the stochastic flow
$\phi_t$ preserves $\mu$ almost surely for all $t$.

\end{example}

\begin{example}
Let $(M,g)$ be a Riemannian manifold that admit a basis of orthonormal vector fields $\{X_1,\ldots, X_n\}$ such that
\[
[X_i,X_j]=\sum_{l=1}^na_{ij}^lX_l.
\]
 Consider the flow $\phi_t$ associated to the SDE
\[
 dx_t=\sum_{i=1}^nX_i(x_t)~\circ dB_t.
\]
Then, the volume measure $\mu$ is invariant if $div_\mu(X_i)=0$ for each $i=1\ldots n$ or equivalently, if $\sum_{j=1}^na_{ji}^j=0$ for each $i=1\ldots n$. In fact, we observe that
 \begin{eqnarray*}
  div_\mu(X_i)&=&\sum_{j=1}^n (\nabla_{X_j}X_i,X_j)\\
  &=&\sum_{j=1}^n\frac{1}{2} ([X_j,X_i],X_j)\\
  &=&\frac{1}{2} \sum_{j=1}^n ( a_{ji}^k X_k, X_j)\\
  &=&\frac{1}{2} \sum_{j=1}^na_{ji}^j.
 \end{eqnarray*}

\end{example}

In the same way, we can study when a current $T_\nu$ is invariant in mean under $\varphi_t$.

\begin{theorem}
 Let $\nu$ and $\mu_g$ as above. The current $T_\nu$ is invariant in mean under $\varphi_t$ if
 \[
 div(fX_0)+\frac{1}{2}\sum_{i=1}^n(X_i+div_{\mu_g}(X_i))(div_{\mu_g}(fX_i))=0.
\]
\end{theorem}
\begin{proof}
 Let $h$ be a smooth function. We have that
 \begin{eqnarray*}
  T_{\nu}\left(X_0h - \frac{1}{2}\sum_{i=1}^nX_i^2h\right)&=&\int_MX_0h~\nu -\frac{1}{2}\sum_{i=1}^n\int_MX_i^2h~\nu.\\
 \end{eqnarray*}
For the second integral above, we observe that
 \begin{eqnarray*}
 \int_MX^2h~\nu&=&\int_Mh~L_X^2\nu  \\
 &=&\int_Mh [X^2f+2Xf~div_{\mu_g}(X)+f(X(div_{\mu_g}(X))+(div_{\mu_g}(X))^2)]~\mu_g\\
 &=&\int_M  h (X+div_{\mu_g}(X))(div_{\mu_g}(fX))  ~\mu_g.
 \end{eqnarray*}
Therefore, $T_{\nu}\left(X_0h -\frac{1}{2}\sum_{i=1}^nX_i^2h\right)=0$ if
\[
 -div_{\mu_g}(fX_0)-\frac{1}{2}\sum_{i=1}^n(X_i+div_{\mu_g}(X_i))(div_{\mu_g}(fX_i))=0.
\]

\end{proof}

%%%%%%%%%%%%%%%%

\section{Application to foliations on Homogeneous manifold.}

A measure $\mu$ on a Riemannian manifold $(M,g)$ define a $0-$current $T_\mu$ by

\[
 T_\mu(f)=\int_Mf~\mu,
\]
for every $f\in C^\infty(M)$. In this section, we study the specific case when the measure $\mu$ is related to the geometric structure given by a foliation $\F$ over $M$ associated with the stochastic flow defined by a SDE that respect this structure. In particular, we relate the invariance with the well known concepts of holonomy invariance and harmonic measures.

There are two particular measures associated to a foliation $\F$ on a Riemannian manifold $M$. The first one is related to a dynamical system on transverse direction of $\F$ called the holonomy pseudogroup. The invariant measures for such a dynamical system, when they exist, are called holonomy invariant measures (see, for example, \cite{candel} or \cite{connes}). When the foliation is oriented, these measures are characterized in terms of currents as a positive current $\psi$ such that 
\[
 \psi(div_L(X))=0,
\]
for every vector field $X$ tangent to the leaf (see Candel \cite{candel1} or Connes \cite{connes}). Here $div_L$ is the divergent operator in the leaf direction.

An alternative way of associating a measure to a foliation $\F$ is via
the foliated Brownian motion (see Garnett \cite{Garnett}). This is
a stochastic process whose infinitesimal generator is given by the
Laplace operator in the leaf direction $\Delta_L$. The invariant
measure associated to this stochastic process is called harmonic
measure and is characterized in the following way: a measure
$\mu$ is harmonic if
\[
 \int_M\Delta_L(f)~d\mu=0,
\]
for each smooth function $f$. Harmonic measures can be described in terms of currents as a positive current $\psi$ such that $\psi(\Delta_Lf)=0$ for each smooth function $f$.

It is interesting to observe that holonomy invariant measures
produce harmonic measures, i.e. if $\psi$ is the current
associated to a holonomy invariant measure, then
\[\psi(\Delta_L(f))=0,\] for every smooth function $f$ (see Candel \cite{candel1}, Garnett \cite{Garnett}). Although the
converse not always happens.

In our setup, we can interpret harmonic measures as follows. Let
$\varphi_t$ the stochastic flow associated to the Brownian motion,
then the current $\psi$ associated to the harmonic measure is
invariant in mean under $\varphi_t.$ We want to find conditions,
for the case of a particular kind of foliations on homogeneous
spaces, such that the current $\psi$ is also invariant under
$\varphi_t$.

Let $G$ be a Lie group and $K$ be a closed subgroup of $G$ with
cofinite volume, i.e. $M\simeq G/K$ is a compact homogeneous
manifold with a invariant metric $\langle ~ ,\rangle$. Since
$M$ is compact, there is a probability invariant measure $\nu$,
such that
\[
 \int_{M} f(g  x)~\nu(dx)=\int_{M} f(x)~\nu (dx),
\]
%and
%\[
% \int_G F(g)~\mu_G(dg)=\int_{G/K}\left(\int_K F(L_g h)~\mu_K(dh)\right)~\nu(d(gK)).
%\]
for every %$F \in C^0(G)$ and
$f \in C_c(M)$ (continuous real-valued
functions with compact support). See e.g. Abbaspour and Moskowitz
\cite[Chap. II, p.106]{abbaspour} among others.

Let $\mathfrak{h}$ be a Lie subalgebra of $\mathfrak{g}$  with
orthonormal basis $\{v_1,\ldots, v_r\}$ and consider the
integrable subbundle $E$ of $M$ by $\{V_1^*,\ldots, V_r^*\}$
defining a foliation $\F$ on $M$. 
Here, we used the fact that
	each $v \in \g$ induce a vector field $V^*$ on $M$ by the formula
	$$V^*_{(gK)} = \frac{d}{dt}\bigg|_{t=0} \exp(tv) \cdot (gK).$$
	Let $\nabla$ the Levi-Civita
connection on $M$ and $\nabla^E$ the connection induced on $E.$
With this notation, the divergence operator in the leaf direction and the Laplace operator in the leaf direction are given by
\[
 div_L(X)=\sum_{i=1}^rg(\nabla^{E}_{V_i^*}X,V_i^*),
\]
for every vector field $X$ tangent to the foliation, and
\[\Delta_E=\sum_{i=1}^r V_i^*V_i^* -\nabla_{V_i^*}^EV_i^*.\]

\begin{lemma}We have that
 \[
 \nabla_{V_i^*}^EV_i^*=-\sum_{k=1}^r c_{ik}^iV_k^*,
\]
where the $c_{ij}^l$ are the structure constants in terms of basis
$\{V_1^*,\ldots, V_r^*\}$ satisfying
\[
[V_i^*, V_j^*]=\sum_{l=1}^r c_{ij}^l V_l^*.
\]
\end{lemma}
\begin{proof} We have by Levi-Civita connection that

 \begin{eqnarray*}
2<\nabla_{V_i^*}^EV_j^*,V_k^*>&=& 2<\nabla_{V_i^*}V_j^*,V_k^*>\\
&=&<[V_i^*,V_j^*],V_k^*>-<[V_j^*,V_k^*],V_i^*>-<[V_i^*,V_k^*],V_j^*>\\
&=&(c_{ij}^k-c_{jk}^i-c_{ik}^j).
 \end{eqnarray*}
Therefore,
\[
 \nabla^E_{V_i^*}V_i^*=-\sum_{k=1}^rc_{ik}^iV_k^*.
\]

%\textcolor{green}{$c_{ii}^k=0$..because..$[V_i^*,V_i^*]=0$...ok!!}
% cuando faz a conta te fica
% \begin{eqnarray*}
%2<\nabla_{V_i^*}^EV_i^*,V_k^*>&=& (c_{ii}^k-2c_{ik}^i)=2c_{ik}^i
% \end{eqnarray*}
%portanto
%\[<\nabla_{V_i^*}^EV_i^*,V_k^*>=c_{ik}^i\]
%
%
\end{proof}

\begin{lemma}\label{lemaintegral}
\[\int_M (V^*f)(x) ~ \nu (dx)=0\hspace{.5cm} for~ all \hspace{.5cm}f \in C^{\infty}(M).\]
\end{lemma}
\begin{proof}
Since
\begin{eqnarray}
(V^*f)(x)&=&  \frac{d}{dt}\bigg|_{t=0} f(\exp(tv) (x)), \nonumber
       %&=& \frac{d}{dt}\bigg|_{t=0} f (L_{\exp_e(tV)}(gK)) . \nonumber
\end{eqnarray}
we have that
\begin{eqnarray}
\int_M (V^*f)(x) ~ \nu (dx) &=& \frac{d}{dt}\bigg|_{t=0} \int_M f(\exp(tv) (x)) ~\nu (dx) \nonumber \\
&=& \frac{d}{dt}\bigg|_{t=0} \int_M  f(x) ~\nu (dx) \nonumber \\
 &=& 0. \nonumber
\end{eqnarray}
\end{proof}

With this, we can prove the following theorem.
\begin{theorem} The measure $\nu$ is a harmonic measure for the induced foliation.
\end{theorem}
\begin{proof}
 Let $f\in C^\infty(M)$. Then,
 \begin{eqnarray*}
 \int_M\Delta_Ef(x)~\nu (dx)&=&\sum_{i=1}^r\int_M\nabla^E_{V_i^*}V_i^*f(x)~\nu(dx)\\
 &=&- \sum_{i,k ~=1}^r c_{ik}^i\int_MV_k^* f(x)~\nu(dx)\\
 &=&0.
\end{eqnarray*}
\end{proof}

\begin{remark}
 Due to the formula above and the Dynkin formula, we can describe the foliated Brownian motion as the solution $\varphi_t$ of the following SDE
 \[
  dx_t=\frac{1}{2}\sum_{k=1}^rc_{ik}^iV_k^*(x_t)~dt+\sum_{i=1}^rV_i^*(x_t)~_\circ dB_t^i,
 \]
 where each vector field $V_k^*$ is defined on compact manifold M.
Thus, the theorem above can be seen as
\[
 \me[\varphi_t^*T_\nu]=T_\nu.
\]
%\textcolor{green}{muito legal!!!}
\end{remark}

\begin{theorem} 
 The measure $\nu$ is totally invariant if and only if $Tr_\mathfrak{h}ad(v_i)=0$ for all $v_i\in \mathfrak{h}$ (where $Tr_\mathfrak{h}$ denotes the trace operator with respect the subspace $\mathfrak{h}$).

 Moreover, if $Tr_\mathfrak{h}ad(v_i)=0$ for all $v_i\in \mathfrak{h}$ then, a measure $\mu$ on $M$ is totally invariant if and only if $T_\mu$ is invariant under the flow of the foliated Brownian motion $\varphi_t$.
\end{theorem}
\begin{proof}
Firstly, we show that $\nu$ is holonomy invariant. Let $X^*=\sum_{i=1}^r f_iV_i^*$, with each $f_i \in C^{\infty
}(M)$. Then,
\begin{eqnarray*}
T_\nu(\textrm{div}_E(X^*))&=&\int_M\textrm{div}_E(X^*)(x)~\nu (dx)
\\&=&\sum_{i=1}^r\int_M
(V_i^*f_i)(x)~\nu(dx)+\sum_{i=1}^r \int_M
(f_i\textrm{div}_E(V_i^*))(x)~\nu(dx).
\end{eqnarray*}
From definition of divergence, we have that
\begin{eqnarray*}
 \textrm{div}_E(V_i^*)&=&\sum_{j=1}^r \langle \nabla_{V_j^*}^EV_i^*,V_j^* \rangle \\
 &=&\sum_{j=1}^r \langle \nabla_{V_j^*}V_i^*,V_j^*  \rangle \\
  &=&-\sum_{j=1}^r \langle V_i^*,\nabla_{V_j^*}V_j^* \rangle \\
  &=&\sum_{j=1}c_{ji}^j\\
  &=&\textrm{Tr}_{\mathfrak{h}}(ad(v_i)).
\end{eqnarray*}
Thus,
\[
 \int_M\textrm{div}_E(X^*)(x)~\nu(dx)=\sum_{i,j~=1}^r \int_M  \textrm{Tr}_{\mathfrak{h}}(ad(v_i)) f_i(x)~\nu(dx).
\]
Where the first part follows.

For the second part, assume that
$\textrm{Tr}_{\mathfrak{h}}(ad(v_i))=0$. Let $f$ be a smooth
function, then by the invariance of $T_\mu$, we obtain that
\[
 V_i^*T_\mu(f)=-\int_MV_i^*f~d\mu=0
\]
Repeating the calculations above with $\mu$ in the place of $\nu$, we obtain the desired.

\end{proof}

\begin{corollary}
If $\mathfrak{h}$ is a semisimple  Lie algebra then $\nu$ is totally invariant.
\end{corollary}

\begin{remark}
 We observe that the same result is valid if $G/K$ is replaced by a smooth manifold $M$ with a Lie
  group transitive action $G\times M\rightarrow M$ that foliates $M$, a invariant measure $\nu$ and a invariant metric $\langle ~,
  \rangle$ on $M$.
\end{remark}

\begin{corollary}
 If $G$ is nilpotent then $\nu$ is totally invariant.
\end{corollary}
\begin{proof}
 If $G$ is nilpotent, then any subalgebra of $\mathfrak{g}$ is nilpotent. Therefore, $Tr_\mathfrak{h}ad(v_i)=0$ for all $v_i\in \mathfrak{h}$.
\end{proof}

\begin{example}
Consider a compact Riemannian homogeneous 3-dimensional manifold
$G/H$, where the associated Lie algebra $\g$ of $G$ is equipped
with a orthonormal basis $\{X,Y,Z \}$ satisfying
\[
[X,Y]=2Y; \ \ \ \ [X,Z]=-2Z; \ \ \ \ [Y,Z]=X.
\]
This Lie algebra is isomorphic to $\mathfrak{sl}(2,\R)$. And, for
the Lie subalgebra $\h$ of closed subgroup $H$ we consider the
vector space generated by $\{ X,Y \}$. Here $\F$ is the foliation
induced by $\h$. We have that $ad(X)$ and $ad(Y)$ with respect to
the basis $\{X,Y \}$ of $\h$ are given by
$$
[ad(X)]_{\mathfrak{h}}=\left(
\begin{array}{cc}
  0 & 0 \\
  0 & 2 \\
\end{array}
\right); \ \ \ [ad(Y)]_{\mathfrak{h}}=\left(
\begin{array}{cc}
  0 & -2 \\
  0 & 0 \\
\end{array}
\right).
$$
Since $\textrm{Tr}_{\mathfrak{h}}(ad(X))=2$ the volume measure defined by the orthonormal basis is not totally invariant.
\end{example}

\begin{example}

For a foliation on a compact manifold $M$ that is obtained by a
subalgebra of the Lie algebra of the transitive group $G$ acting
on $M$, the volume measure is harmonic and also holonomy invariant
by the theorem above.

In particular this is the case of a foliation on a Heisenberg manifold, when viewed as a circle bundle over the $2-$torus, where the leaves are the fibers of the bundle.

\end{example}

%\textcolor{red}{Precisa citar um artigo que fala sobre measure
%totally invariant???  }


\begin{thebibliography}{99}
\bibitem{abbaspour}  Abbaspour, H. and Moskowitz, M. {\it Basic Lie Theory } World Scientific, Singapore, 2007.

\bibitem{candel1} Candel, A. \textit{The harmonic measures of Lucy Garnett}. Adv. Math.  176  (2003),  2, 187-247.

\bibitem{candel} Candel, A. and Conlon, L. \textit{Foliations I}. Graduate Studies in Mathematics, 23. American Mathematical
Society. \textit{Foliations II}. Graduate Studies in Mathematics, 60. American Mathematical Society.


\bibitem{ana}  Cannas da Silva, A. {\it Lectures on symplectic geometry } Lecture Notes in Math.,
1764. Springer-Verlag, Berlin, 2001.
%\bibitem{adams}S. R. Adams, \textit{Superharmonic functions on foliations}. Trans. Amer. Math. Soc. 330 (1992), 2, 625-635.


\bibitem{connes}  Connes, A. \textit{A survey of foliations and operator algebras.} Operator algebras and applications, Part I (Kingston, Ont., 1980), pp. 521–628, Proc. Sympos. Pure Math., 38, Amer. Math. Soc., Providence, R.I., 1982.

\bibitem{derham}  de Rham, G. \textit{ Differentiable manifolds. Forms, currents, harmonic forms.} Springer-Verlag, Berlin, 1984.
%\bibitem{pedro} P. J. Catuogno, \textit{A geometric It\^o formula}. Matem\'atica Contempor\^anea, 2007, vol. 33, p. 85-99.
%\bibitem{agrachev} A. Agrachev; Y. L. Sachkov. \textit{Control theory from the geometric viewpoint.} Encyclopaedia of Mathematical Sciences, 87. Springer-Verlag, Berlin, 2004.
%\bibitem{colonius} F. Colonius; W. Kliemann \textit{The dynamics of control}.  Birkhäuser Boston, Inc., Boston, MA, 2000.
%\bibitem{deroin}B. Deroin,V. Kleptsyn \textit{Random conformal dynamical systems.} Geom. Funct. Anal. 17 (2007), Nro. 4, 1043–1105.
\bibitem{Elworthy} Elworthy, K.D. \textit{Stochastic Differential Equations on Manifolds}. London Math. Society (Lecture Notes Series 70) Cambridge
University Press 1982
%\bibitem{emery} M. Emery, \textit{Stochastic Calculus in Manifolds}. Universitext. Springer-Verlag,. Berlin, 1989.
%\bibitem{feres} S. Fenley, R. Feres and K. Parwani, \textit{Harmonic functions on $\mr$ covered foliations and group actions on the circle}
%to appear in Ergodic Theory and Dynamical Systems.
%\bibitem{feres2}R. Feres and A. Zeghib, \textit{Dynamics on the space of harmonic functions and the foliated Liouville problem}. Ergodic theory and dynamical %systems, Vol. 25, 2 (2005), pp. 503-516
\bibitem{Garnett} Garnett, L.  \textit{Foliation, The ergodic theorem and Brownian motion}. Journal of Functional Analysis 51, (1983)pp. 285-311
%\bibitem{ghys} E. Ghys, \textit{Gauss-Bonnet theorem for 2-dimensional foliations}, J. Functional Analysis 77 (1988) pp. 51-59.
%\bibitem{Greene}Greene, R. E. \textit{Isometric embeddings}.
%Bull. Amer. Math. Soc. Volume 75, Number 6 (1969), 1308-1310.
%\bibitem{gromov2} M. Gromov, \textit{Foliated Plateau problem, part II: Harmonic maps of foliations}. Geometric And Functional Analysis. Vol. 1, Number 3(1991)
\bibitem{IW} Ikeda, N. and Watanabe, S. \textit{Stochastic differential equations and diffusion processes}. Second edition. North-Holland Mathematical Library, 24.
%North-Holland Publishing Co., Amsterdam; Kodansha, Ltd., Tokyo, 1989.
%\bibitem{hsu} E. P. Hsu, \textit{Stochastic Analisis on Manifolds}. Graduate Studies in Mathematics. AMS.
%\bibitem{kaimainovich} V. A. Kaimanovich,\textit{Brownian motion on foliations: Entropy, invariant measures, mixing}
%Functional Analysis and Its Applications, Vol. 22, N$^0$ 4, (1988) pp. 326-328
%\bibitem{kobayashi nomizu} S. Kobayashi and K. Nomizu, \textit{Foundations of Differential Geometry}, vol I, Interscience  Publishers, New York 1963.
%\bibitem{krener} A. J. Krener, \textit{A generalization of Chow's theorem and the bang-bang theorem to non-linear control problems}. SIAM J. Control 12 (1974), pp. 43-52.
%\bibitem{kunita}  H. Kunita, \textit{Stochastic  flows and stochastic differential equations}. Cambridge Univ. Press. (1997).









\bibitem{mcdonald} Kinateder, K. and McDonald, P. \textit{An Ito formula for domain-valued processes driven by stochastic flows.} Probab. Theory Related Fields 124 (2002), no. 1, 73–99
\bibitem{klieman} Kliemann, W. \textit{Recurrence and invariant measures for degenerate diffusions.} Ann. Probab. 15 (1987), no. 2, 690–707.

\bibitem{Kunita-1}  Kunita, H. -- { Stochastic differential equations and
stochastic flows of diffeomorphisms,}  { in \'{E}cole d'Et\'{e} de
Probabilit\'{e}s de Saint-Flour XII }- 1982, ed. P.L. Hennequin,
pp. 143--303. Lecture Notes on Maths. 1097 (Springer, 1984).


\bibitem{Kunita-2}  Kunita, H. -- { Stochastic flows and stochastic
differential equations}, Cambridge University Press, 1988.

\bibitem{ustunel} Ustunel, A. S. \textit{A generalization of Ito's formula.} J. Funct. Anal. 47 (1982), no. 2, 143–152.

%\bibitem{litam}P. Li, L. Tam, \textit{The heat equation and harmonic maps of complete manifolds}.
%Invent. math. 105 (1991).
%\bibitem{norris}  J. R. Norris, \textit{A complete differential formalism for stochastic calculus on manifolds}.
%S\'eminaires de Probabilit\'es XXVI, 189-209. Lect. Notes in Math. vol. 1526, Springer 1980.
%\bibitem{ledrappier} F. Ledrappier, \textit{Ergodic properties of the stable foliations}, Lecture Notes in Math., Vol. 1514, Springer, Berlin, (1993) pp. 131-145.
%\bibitem{PM} P. Liu and M. Qian \textit{Smooth ergodic theory of Random Dynamical Systems}.
%Lecture Notes in Mathematics 1606, Springer.
%\bibitem{Molino} P. Molino, \textit{Riemannian foliations}. Progress in Mathematics, 73. Birkhauser Boston, Inc., Boston, MA, 1988.
%\bibitem{plante}J. F. Plante, \textit{Foliations with measure preserving holonomy}. Annals of  Mathematics, 102 (1975), 327-361.
%\bibitem{protter} P. Protter, \textit{Stochastic Integration and Differential Equations}. Springer-Verlag (1990).
%\bibitem{Rumler} H. Rumler, \textit{Differential forms, Weitzembock formulae and foliations}. Publicacions Matematiques, Vol 33 (1989), pp. 543-554.
%\bibitem{stefan} P. Stefan. \textit{Accessible Sets, Orbits and Foliations with Singularities}. Proc. London Math. Soc. (3) 29 (1974), pp. 699-713.
%\bibitem{Tondeur} P. Tondeur \textit{Foliations on Riemannian manifolds}. Universitext, Springer Verlag, Berlin-Heidelberg-New York, 1988.
%\bibitem{yue} Ch. Yue. \textit{Brownian motion on Anosov foliations and manifolds of negative curvature}  J. Differential Geom. Vol. 41, N$^0$ 1 (1995), %pp.159-183.
%\bibitem{Yosida} K. Yosida,\textit{Functional analysis}. Sixth edition. Springer-Verlag, Berlin-New York, 1980.

\end{thebibliography}
\end{document}